\documentclass[12pt]{article}
\usepackage{mathrsfs}
\usepackage{amsthm}
\usepackage{amssymb}
\usepackage{latexsym}
\usepackage{amsmath,amsfonts}
\usepackage{cases}
\usepackage{bm}
\usepackage{indentfirst}
\usepackage{xcolor}
\usepackage{ifpdf}
\usepackage{graphicx}
\usepackage{epstopdf}
\usepackage{epsfig}
\usepackage{psfrag}
\usepackage{enumitem}
\usepackage{verbatim}
\usepackage{color}
\usepackage{subfigure}

\usepackage[
pdfauthor={Lu},
pdftitle={matching covered},
pdfstartview=XYZ,
bookmarks=true,
colorlinks=true,
linkcolor=blue,
urlcolor=blue,
citecolor=blue,
bookmarks=true,
linktocpage=true,
hyperindex=true
]{hyperref}

\usepackage[textwidth=16cm, textheight=22cm, includehead, includefoot]{geometry}

\title{The maximum number of edges in minimal matching covered graphs \footnote{The research is supported by NSFC (Grant Nos. 12671409, 12271235) and NSF of Fujian Province (Grant No. 2026J002034),  Fujian Key Laboratory of Granular Computing and Applications and Fujian Key Laboratory of Data Science and Statistics.
\newline E-mail addresses: xiaolinghe99@163.com (X. He), flianglu@163.com (F. Lu), zhanghp@lzu.edu.cn(H. Zhang).} }
\author{ Xiaoling He $^1$, Fuliang Lu $^2$, Heping Zhang $^1$\\
\small {1. School of Mathematics and Statistics, Lanzhou University, Lanzhou, China}\\
\small {2. School of Mathematics and Statistics, Minnan Normal University, Zhangzhou, China}
}
\date{}

\newtheorem{lem}{Lemma}[section]
\newtheorem{thm}[lem]{Theorem}
\newtheorem{cor}[lem]{Corollary}

\newtheorem{pro}[lem]{Proposition}

\newtheorem{fact}[lem]{Fact}

\newtheorem{cla}{Claim}
\newtheorem*{claim}{Claim}

\begin{document}
\bibliographystyle{plain}
\newcommand{\udots}{\mathinner{\mskip1mu\raise1pt\vbox{\kern7pt\hbox{.}}
\mskip2mu\raise4pt\hbox{.}\mskip2mu\raise7pt\hbox{.}\mskip1mu}}
\maketitle
\begin{abstract}
 A connected  graph $G$ with at least two vertices is {\em matching covered} if each of its edges lies in a perfect matching.
A matching covered graph is {\em minimal} if the removal of any edge results in a graph that is no longer matching covered.
Lov\'asz and Plummer [J. Combin. Theory, Ser. B  23 (1977) 127–138] proved by ear decompositions that every minimal matching covered bipartite graph $G$ with at least 6 vertices has at most  $(3|V(G)|-6)/2$ edges, and this bound is sharp for all even $|V(G)|\ge6$.  Very recently, Mallik,  Diwan and  Kothari [J. Combin. 17 (4) (2026) 559–606] determined the extremal  graphs achieving this upper bound.

In this paper, we prove that every minimal matching covered nonbipartite graph $G$ with at least 6 vertices has at most $(5|V(G)|-10)/2$ edges.  Moreover, we show that if $G$  is not bicritical, then $G$ has at most  $(5|V(G)|-14)/2$ edges and if $G$ is 3-connected and has   at least 14 vertices, then $G$ has at most $(5|V(G)|-26)/2$ edges. All three bounds are sharp, and all the corresponding extremal graphs are characterized. 
\end{abstract}
\vskip 0.2in \baselineskip 0.1in

\section{Introduction}

Graphs considered in this paper may have multiple edges, but no loops. We follow \cite{BM08} for undefined notation and terminology. 
Let $G$ be a graph with the vertex set $V(G)$ and the edge set $E(G)$. For a vertex $u\in V(G)$, let $N_G(u)$, or simply $N(u)$, denote the set of vertices of $G$ adjacent to $u$, and let $d_G(u)$, or simply $d(u)$, denote the degree of $u$ in $G$, that is, the number of edges of $G$ incident with $u$. We denote by $\delta(G)$ and $\Delta(G)$ the {\em minimum degree} and the {\em maximum degree} of $G$, respectively.  
For $\emptyset\neq X \subset V (G)$, let $N_G(X)$ denote the set of all vertices in $V (G) \setminus X$ which are adjacent to at least one vertex in $X$, and let $G[X]$ be the subgraph of $G$ induced by $X$.

A {\em perfect matching} $M$ of a graph $G$ is a set of edges such that each vertex of $G$ is incident with exactly one edge of $M$. 
An edge $e$ of a graph $G$  is {\em allowed} if there exists some perfect matching of $G$ containing $e$.
A connected nontrivial graph $G$ is {\em matching covered} if each of its edges is allowed.
We say that an edge $e$ in a matching covered graph $G$ is {\em removable} if $G-e$ is matching covered.
We say that a matching covered graph $G$ is {\em minimal} if $G-e$ is not a matching covered graph for any edge $e$ in $G$, equivalently, $G$ contains no removable edges.  Every minimal matching covered graph is simple because each multiple edge is removable. 

For a minimal matching covered bipartite graph with at least 6 vertices\footnote{Lov\'{a}sz and Plummer used the terminology ``minimal elementary bipartite graph'' in \cite{LP77}.},  Lov\'{a}sz and Plummer \cite{LP77} presented an upper bound for $|E(G)|$.
\begin{thm}[\cite{LP77}]\label{thm:minimal-bi-MC}
    Every minimal matching covered bipartite graph $G$ with at least 6 vertices  has at most $(3|V(G)|-6)/2$ edges.
\end{thm}
Moreover, Lov\'{a}sz and Plummer \cite{LP77} showed that the upper bound of Theorem \ref{thm:minimal-bi-MC} is sharp for all  even $|V(G)| \ge 6$.
Very recently,  Mallik, Diwan and Kothari \cite{MDK2026}  characterized minimal matching covered bipartite graphs $G$ with $|E(G)|=(3|V(G)|-6)/2$ in terms of special trees.
In this paper we focus on the maximum number of edges in minimal matching covered nonbipartite graphs. The following are our main results.

For any integer $k\ge2$, let ${G}_k$ be the graph obtained from a matching $\{u_iv_i:1\le i\le k\}$ and two additional vertices $a$ and $b$ by adding the edges $\{au_i,av_i,bu_i,bv_i:1\le i\le k\}$ (see Figure \ref{fig:sharp-exp} when $k=3$). 
\begin{thm}\label{thm:main-thm}
    Let $G$ be a minimal matching covered nonbipartite graph  with at least 6 vertices. Then $|E(G)|\le (5|V(G)|-10)/2$. Moreover, equality holds if and only if $G\in\{{G}_k:k\ge2\}$. 
\end{thm}
\begin{figure}[!h]
\centering
\includegraphics[totalheight=3.5cm]{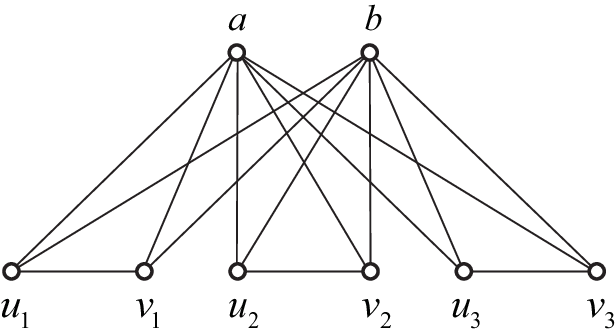}
\caption{The graph $G_3$.}
\label{fig:sharp-exp}
\end{figure}

A graph $G$ with four or more vertices is {\em bicritical} if for any two distinct vertices $u$ and $v$ in $G$, $G-\{u, v\}$ has a perfect matching. Obviously, every bicritical graph is matching covered. 
In fact, every graph of $\{{G}_k:k\ge2\}$ is bicritical. Let ${H}_k:={G}_k-av_1-bu_1$. 
For minimal matching covered nonbipartite graphs that are not bicritical, we have the following result.  

\begin{thm}\label{thm:nonbicritical}
    Let $G$ be a minimal matching covered nonbipartite graph with $|V(G)|\ge6$. If $G$ is not bicritical, then $|E(G)|\le (5|V(G)|-14)/2$. Moreover, equality holds if and only if $G\in\{{H}_k:k\ge2\}$.
\end{thm}

Every matching covered graph with at least 4 vertices is 2-connected (by Corollary \ref{cor:MC_2-con}). For a 3-connected minimal matching covered nonbipartite graph $G$, what is the maximum number of edges of $G$? 
For a minimal matching covered graph $G$ other than $K_2$, Lov\'{a}sz and Plummer \cite{LP77} proved that $\delta(G)=2$ if $G$ is bipartite, and He, Lu and Xue \cite{HLX2026} proved that $\delta(G)\in\{2,3\}$ if $G$ is nonbipartite. Thus, every minimal matching covered graph has connectivity at most 3 and every 3-connected minimal matching covered graph is nonbipartite. 
Moreover, we have the following result.

For any integer $k\ge2$,  let $\mathcal{G}_k$ be the family of graphs obtained from ${G}_k-b$, two disjoint triangles $T_1:=t_1t_2t_3t_1$ and $T_2:=t'_1t'_2t'_3t'_1$, and the three additional vertices  $\{s_1,s_2,s_3\}$ (say $B$) by adding the edges $\{t_js_j,t'_js_j:1\le j\le 3\}$ and joining each vertex of $\{u_i,v_i:1\le i\le k\}$ to exactly one vertex of $B$,  such that $N(u_i)\cap B\neq N(v_i)\cap B$ for each $1\le i\le k$, and $N(s_j)\cap \{u_i,v_i:1\le i\le k\}\neq\emptyset$ for each $1\le j\le 3$. 
Note that for each $k\ge2$, $\mathcal{G}_k\neq\emptyset$ and every graph in $\mathcal{G}_k$ is simple (see Figure \ref{fig:sharp-exp-3-connected} for an example in $\mathcal{G}_3$).

\begin{figure}[!h]
\centering
\includegraphics[totalheight=3.5cm]{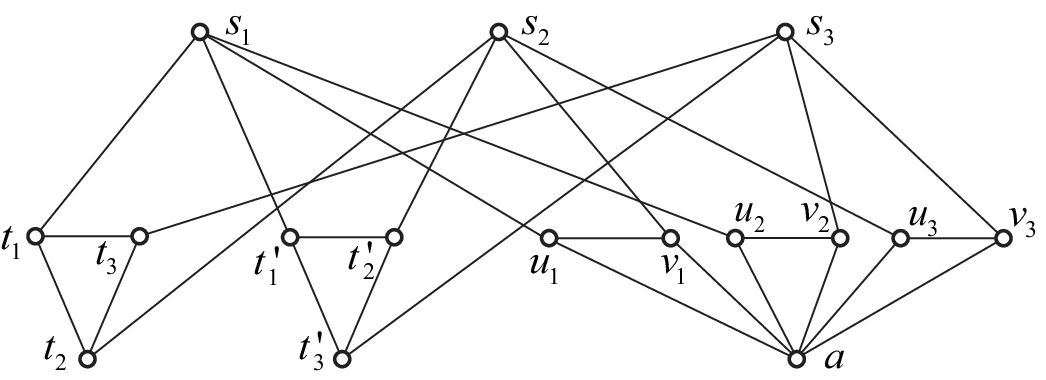}
\caption{A graph in $\mathcal{G}_3$.}
\label{fig:sharp-exp-3-connected}
\end{figure}
 
\begin{thm}\label{thm:main2}
    Let $G$ be a 3-connected minimal matching covered graph with at least 14 vertices. Then $|E(G)|\le (5|V(G)|-26)/2$. Moreover, equality holds if and only if $G\in \cup_{k\ge2}\mathcal{G}_k$.  
\end{thm} 

The condition $|V(G)|\ge14$ of Theorem \ref{thm:main2} is necessary. 
Let $G_0$ be the graph obtained from three disjoint triangles and three additional vertices $\{s_1,s_2,s_3\}$ by joining the three vertices of each triangle to $s_1,s_2,s_3$, respectively. 
Then $G_0$ is a 3-connected cubic graph and hence is matching covered \cite{Tutte47}.  
Let $t_1t_2t_3t_1$ be any of the three triangles. 
For each $i\in\{1,2,3\}$ (with subscripts taken modulo 3), every perfect matching of $G_0$ contains $t_it_{i+1}$ if and only if it contains the unique edge joining $t_{i+2}$ to $\{s_1,s_2,s_3\}$. 
Thus, these two edges are not removable. 
Since the triangle is arbitrary, every edge of $G_0$ is nonremovable, and hence $G_0$ is minimal matching covered. 
However, $|V(G_0)|=12$ and $|E(G_0)|=18>(5|V(G_0)|-26)/2=17$.

In Section 2, we present some notation and preliminary results. In Sections 3, 4 and 5, we will prove  Theorems \ref{thm:main-thm}, \ref{thm:nonbicritical} and  \ref{thm:main2}, respectively.

 \section{Preliminaries}

 A set of vertices in a graph $G$ is said to be {\em independent} if no two of them are adjacent. 
A component with an odd (even) number of vertices in a graph is called an {\em odd (even) component}. A component is {\em trivial} if it contains exactly one vertex.
We denote by $o(G)$ the number of odd components of the graph $G$. 
 If a graph $G$ has a perfect matching, then a vertex subset $B\subseteq V(G)$ is called a {\em barrier} of $G$ if $o(G-B)=|B|$.  
A barrier is {\em trivial} if it contains at most one vertex. 
Tutte  proved the following classical theorem in 1947.
\begin{thm}[\cite{Tutte47}]\label{thm:Tutte}
    A graph $G$ has a perfect matching if and only if $o(G-S)\le|S|$, for every $S\subseteq V(G)$.
\end{thm}

Using Theorem \ref{thm:Tutte}, we have the following properties about matching covered graphs.
\begin{cor}[\cite{Lovasz87}]\label{cor:M-C-without-even-components}
  Let $G$ be a matching covered graph. Every nonempty barrier $B$ is independent, and $G-B$ has no even components.
\end{cor} 
\begin{cor}[{\cite[Corollary 2.7]{Lucchesi2024}}]\label{cor:MC_2-con}
    Every matching covered graph on four or more vertices is 2-connected.
\end{cor} 

Denote by $G[A,B]$ the bipartite graph with bipartition $(A,B)$.
\begin{lem}[\cite{Hetyei1964}, {\cite[Theorem 4.1.1]{LP86}}]\label{lem:char_MC_bi}
Let $G[U,W]$ be a connected bipartite graph with a perfect matching and at least four vertices. 
Then $G$ is matching covered if and only if $G-u-w$ has a perfect matching for every $u\in U$ and $w\in W$.
\end{lem}

\begin{pro}[{\cite[Theorem 5.2.5]{LP86}}]\label{pro:mc_is_Bi}
    A matching covered graph $G$ with at least 4 vertices  is bicritical if and only if every barrier of $G$ is trivial.
\end{pro}

A bicritical graph $G$ is {\em minimal} if $G-e$ is not bicritical for every $e\in E(G)$.  A minimal bicritical graph need not be minimal matching covered. Moreover, we have the following proposition.

\begin{pro}\label{pro:minimal_MC_bicritical}
    Let $G$ be a minimal matching covered graph that is bicritical. Then $G$ is a minimal bicritical graph.
\end{pro}
\begin{proof}
    Since $G$ is minimal matching covered, for every $e\in E(G)$, $G-e$ is not matching covered. 
    Since every bicritical graph is matching covered, $G-e$ is not bicritical. 
    Therefore, $G$ is minimal bicritical.
\end{proof}

A wheel is a simple graph obtained from a cycle by adding a new vertex and joining this new vertex to each vertex of the cycle. 
\begin{lem}[{\cite[Corollary 5.4.16]{LP86}},\cite{ZWY2022}]\label{lem:minimal_bicritical}
    If $G$ is a minimal bicritical graph with at least 6 vertices, then $|E(G)|\le \frac{5|V(G)|-10}{2}$. Moreover, equality holds if and only if $G$ is a wheel of order 6 or $G\in\{{G}_k:k\ge2\}$ (see the definition of ${G}_k$ in Section 1). 
\end{lem}

\subsection{Tight cuts and nonremovable edges}
For nonempty proper sets $X,Y\subset V(G)$, by $E_G[X,Y]$ we mean the set of edges of $G$ with one end vertex in $X$ and the other end vertex in $Y$. 
We call $\partial_G(X):=E_G[X,\overline{X}]$ an {\em edge cut}, or simply a {\em cut}, of $G$, where $\overline{X}:=V(G)\setminus X$. 
If $X = \{u\}$, then we denote $\partial_G(\{u\})$, for brevity,  by $\partial_G(u)$.  (If $G$ is understood, the subscript $G$ will be omitted.) 
An edge cut $\partial(X)$ is {\em trivial} if $|X| = 1$ or $|\overline{X}| = 1$.
An edge cut $\partial(X)$ is {\em tight} in $G$ if $|\partial(X)\cap M|=1$ for every perfect matching $M$ of $G$. 
 For an edge cut $\partial(X)$ of $G$, let $G/(X\to x)$, or simply $G/X$, denote the graph obtained from $G$ by contracting $X$ to a single vertex $x$ and removing any resulting loops. The graphs $G/X$ and $G/\overline{X}$ are {\em $\partial(X)$-contractions} of $G$. 
 More generally, for pairwise disjoint nonempty vertex sets $X_1,\ldots,X_r\subseteq V(G)$, by $G/X_1/\ldots/X_r$ we mean the graph obtained from $G$ by successively contracting $X_1$, $\ldots$, $X_r$ to singletons.  
 When no confusion arises, we regard the edges of the contractions as their corresponding edges in $G$. 

\begin{lem}[\cite{Lovasz87}]\label{lem:C-contractions-MC}
    Let $G$ be a matching covered graph. If $C$ is a tight cut of $G$, then both $C$-contractions are matching covered. Moreover, if $G'$ is a $C$-contraction of $G$, then a tight cut of $G'$ is also a tight cut of $G$. Conversely, if a tight cut of $G$ is a cut of $G'$, then it is also tight in $G'$.
\end{lem}

 The following simple fact will be used frequently.
\begin{fact}\label{fact:C-contractions-vtx-number}
Let $G$ be a matching covered graph with a nontrivial tight cut $\partial(X)$. Then $|V(G)|=|V(G/X)|+|V(G/\overline{X})|-2$.
\end{fact} 

\begin{pro}[\cite{Lovasz87}]\label{pro:tight-cut-in-bipartiteMC}
    Let $G[U,W]$ be a matching covered graph. An edge cut $\partial(X)$ of $G$ is tight if and only if $||X\cap  U|-|X\cap  W||=1$ and every edge of $\partial(X)$ is incident with a vertex of the larger one between the two sets, $X\cap  U$ and $X \cap W$.
\end{pro}

 A matching covered  graph that is free of nontrivial tight cuts is  a {\em brick} if it is nonbipartite and a {\em brace}  if it is bipartite.
 Moreover, a graph $G$ is a brick if and only if $G$ is 3-connected and bicritical \cite{ELP82}.
 The {\em tight cut decomposition}, due to Lov\'{a}sz \cite{Lovasz87}, can be applied to a matching covered graph  to produce a unique list of bricks and braces up to multiple edges.

 \begin{lem}[\cite{CLM15}]\label{lem:RE_brace}
     Every edge of a brace with at least 6 vertices is removable.
 \end{lem}
 \begin{lem}[\cite{CLM99}]\label{lem:re_in_brick}
   Every brick different from $K_4$ and the triangular prism $\overline{C_6}$ has at least $\Delta(G)-2$ removable edges.
\end{lem}

 A brace (resp., brick) $G$ is {\em minimal} if $G-e$ is not a brace (resp., brick) for any $e\in E(G)$.
 Fabres, Kothari and Carvalho \cite{FKC2020} proved that every minimal brace $G$ with at least 12 vertices has at most $(5|V(G)|-20)/2$ edges. Moreover, they provided a complete characterization of minimal braces that meet this upper bound.
 For every minimal brick $G$ with $|V(G)|\ge10$,  Norine and Thomas \cite{NT2006} proved that $|E(G)|\le(5|V(G)|-14)/2$, and the upper bound is sharp.

The following lemma is easy to verify by the definition of matching covered graphs.
 \begin{lem}[\cite{CLM99}]\label{lem:re_also_re}
    Let $G$ be a matching covered graph with a nontrivial tight cut $C$. Let $G'$ and $G''$ be the two $C$-contractions of $G$. Then for an edge $e$ in $G$, $G-e$ is matching covered if and only if $G'-e$ and $G''-e$ are matching covered.
\end{lem}
 For a matching covered graph $G$, let $E_N(G)$ denote the set of all nonremovable edges of $G$.
The following corollary can be obtained by Lemma \ref{lem:re_also_re} directly.
\begin{cor}\label{cor:NR_TC}
    Let $G$ be a matching covered graph with a nontrivial tight cut $C$, and let $G'$ and $G''$ denote two $C$-contractions of $G$. Then $E_N(G)=E_N(G')\cup E_N(G'')$.
\end{cor}

\subsection{Splicing}
Let $u$ and $v$ be vertices of two vertex-disjoint graphs $G$ and $H$, respectively, such that $d_G(u)=d_H(v)$. Moreover, let $\theta$ be a given bijection between $\partial_G(u)$ and $\partial_H(v)$. We denote by $(G(u)\odot H(v))_\theta$ the graph obtained from the union of $G-u$ and $H-v$ by joining, for each edge $e\in \partial_H(v)$, the end of $e$ in $H$ belonging to $V(H)\setminus\{v\}$ to the end of $\theta(e)$ in $G$ belonging to $V(G)\setminus\{u\}$; and refer to $(G(u)\odot H(v))_\theta$  as the graph obtained by {\em splicing} $G$ (at $u$), with $H$ (at $v$), with respect to the bijection $\theta$, for brevity, to $G(u)\odot H(v)$. In general, the graph resulted from splicing two graphs $G$ and $H$ depends on the choice of $u$, $v$ and $\theta$. The following proposition can be gotten by the definition of matching covered graphs directly (e.g., see Theorem 2.13 in \cite{Lucchesi2024}).  

\begin{pro}[{\cite[Theorem 2.13]{Lucchesi2024}}]\label{pro:MC_IS_MC} 
    The splicing of two matching covered graphs is also matching covered.
\end{pro}  

\section{Proof of Theorem \ref{thm:main-thm}}

Now we will present an upper bound of the number of nonremovable edges in a matching covered graph. An edge cut $C$ of a matching covered graph $G$ is called a {\em barrier cut} if $C=\partial_G(V(Q))$ for some barrier $B$ of $G$ and some odd component $Q$ of $G-B$. Clearly, every trivial tight cut is a barrier cut, and every barrier cut is a tight cut. 
\begin{lem}\label{lem:H_in_MC}
    Let $G$ be a matching covered graph with at least 6 vertices. Then there exists a spanning subgraph $H$ of $G$ such that $E_N(G)\subseteq E(H)$, $\delta(H)\ge1$ and $|E(H)|\le \frac{5|V(G)|-10}{2}$. In particular, if $G$ is not bicritical, then $|E(H)|\le \frac{5|V(G)|-14}{2}$. 
\end{lem}
\begin{proof}
    Assume first that $G$ is bicritical. 
    Let $H$ be a minimal bicritical spanning subgraph of $G$ (it is possible that $H=G$). 
    For any  $e\in E(G)\setminus E(H)$, we will show that $e$ is removable in $G$.   
    For any two distinct vertices $u,v\in V(H)$, $H-\{u,v\}$ has a perfect matching $M$, as $H$ is bicritical. 
    Since $V(H)=V(G)$, $M$ is also a perfect matching of $G-\{u,v\}$. 
    Moreover, as $M\subset E(H)$ and $e\notin E(H)$, we have $e\notin M$ and hence, $M$ is a perfect matching of $(G-e)-\{u,v\}$. 
    So $G-e$ is bicritical and hence matching covered. 
    Thus, $e\in E(G)\setminus E_N(G)$. 
    Therefore, $E_N(G)\subseteq E(H)$. 
    Since $H$ is bicritical, it is matching covered. 
    Thus, $H$ is 2-connected by Corollary \ref{cor:MC_2-con}, and hence $\delta(H)\ge2$. 
    By Lemma \ref{lem:minimal_bicritical}, $|E(H)|\le \frac{5|V(H)|-10}{2}$ as $H$ is a minimal bicritical graph with $|V(H)|=|V(G)|\geq 6$. 
    Thus, the result holds. 

    Now we assume that $G$ is not bicritical. 
    Then $G$ is not a brick. 
    If $G$ is a brace, let $E(H)$ be chosen as a perfect matching of $G$. 
    Clearly, $\delta(H)=1$. Since $|V(G)|\ge6$,
    by Lemma \ref{lem:RE_brace}, $E_N(G)=\emptyset$ and hence $E_N(G)\subset E(H)$, and $|E(H)|=\frac{|V(G)|}{2}<\frac{|V(G)|+(4|V(G)|-14)}{2}=\frac{5|V(G)|-14}{2}$. 
    Thus, the result holds for braces. 
    
    Now we may assume that $G$ has a nontrivial tight cut and we proceed by induction on $|V(G)|$. 
    For every nontrivial tight cut $C$ of $G$, the two $C$-contractions are matching covered by Lemma \ref{lem:C-contractions-MC}. 
    \begin{claim}
        $G$ has a nontrivial tight cut $\partial_G(X)$ such that $G/X$ has a nontrivial barrier.
    \end{claim}
    \begin{proof}
    Since $G$ is not bicritical, $G$ has a nontrivial barrier by Proposition \ref{pro:mc_is_Bi}. 
    If $G$ is nonbipartite, let $B$ be a nontrivial barrier of $G$. 
    By Corollary \ref{cor:M-C-without-even-components}, $B$ is independent and $G-B$ has no even components. 
    Since $G$ is nonbipartite, $G-B$ has a nontrivial odd component $Q$. 
    Let $X:=V(Q)$. 
    Then $\partial_G(X)$ is a nontrivial barrier cut, and $B$ is also a nontrivial barrier of $G/X$. 
    If $G$ is bipartite, choose any nontrivial tight cut $\partial_G(X)$. 
    Then $G/X$ is bipartite and matching covered, so each of its partite sets is a nontrivial barrier. 
    \end{proof}

    By the Claim, we may choose a nontrivial tight cut $\partial_G(X)$ of $G$ such that $G/X$ has a nontrivial barrier. 
    Let $G'_1:=G/(X\to x_1)$ and $G'_2:=G/(\overline{X}\to x_2)$. 
    Then  $|V(G'_1)|+|V(G'_2)|=|V(G)|+2$ and $4\le|V(G'_i)|<|V(G)|$ for each $i\in\{1,2\}$. 

    If $|V(G)|=6$, then $|V(G'_1)|=4=|V(G'_2)|$. By Corollary \ref{cor:MC_2-con}, $G$ is 2-connected. 
    Then there exist two nonadjacent edges $f$ and $f'$ of $\partial(X)$. 
    Hence, $f$ and $f'$ are nonparallel in each of  $\partial_{G'_1}(x_1)$ and $\partial_{G'_2}(x_2)$. 
    Thus, there exists an underlying simple graph $H'_i$ of $G'_i$ that contains $f$ and $f'$ for each $i\in\{1,2\}$. 
    Since multiple edges are removable, we have $E_N(G'_i)\subseteq E(H'_i)$. 
    Since $G'_1$ has a nontrivial barrier, $H'_1$ is a 4-cycle. 
    Note that $H'_2$ is either a 4-cycle or $K_4$. 
    Thus, $\delta(H'_1)=2$, $\delta(H'_2)\ge2$, $|E(H'_1)|=4$ and $|E(H'_2)|\le6$. 
    Let $H$ be the subgraph of $G$ with edge set $E(H'_1)\cup E(H'_2)$.
    Clearly, $V(H)=V(G)$ and $\delta(H)\ge1$.
    Since $E_N(G)=E_N(G'_1)\cup E_N(G'_2)$ by Corollary \ref{cor:NR_TC}, we have $E_N(G)=E_N(G'_1)\cup E_N(G'_2)\subseteq E(H'_1)\cup E(H'_2)=E(H)$.
    Note that $|E(H'_1)|+|E(H'_2)|-|E(H'_1)\cap E(H'_2)|\le 4+6-2=8$.  
    Therefore, the result follows for $|V(G)|=6$. 

    Now suppose that $|V(G)|\ge8$, and that the result holds for all such graphs with fewer vertices than $G$. 
    Since $G'_1$ has a nontrivial barrier, $G'_1$ is not bicritical by Proposition \ref{pro:mc_is_Bi}. 
    If $|V(G'_1)|\ge6$, then by inductive hypothesis, $G'_1$ has a spanning subgraph $H'_1$ such that $E_N(G'_1)\subseteq E(H'_1)$, $\delta(H'_1)\ge1$ and $|E(H'_1)|\le(5|V(G'_1)|-14)/2$. 
    If $|V(G'_2)|\ge6$, we apply the bicritical case proved above when $G'_2$ is bicritical, and the inductive hypothesis otherwise. In both cases, $G'_2$ has a spanning subgraph $H'_2$ such that $E_N(G'_2)\subseteq E(H'_2)$, $\delta(H'_2)\ge1$, and $|E(H'_2)|\le(5|V(G'_2)|-10)/2$.

    Since $|V(G)|\ge8$, at least one of $G'_1$ and $G'_2$ has at least 6 vertices. 
    If $|V(G'_1)|\ge6$ and $|V(G'_2)|\ge6$, then $|E(H'_1)\cup E(H'_2)|\le(5|V(G'_1)|-14)/2+(5|V(G'_2)|-10)/2=(5|V(G)|-14)/2$. 
    Now we consider the case when exactly one of $G'_1$ and $G'_2$ has exactly four vertices. 
    We may choose $j\in\{1,2\}$ such that  $|V(G'_j)|=4$.
    Then $|V(G'_{3-j})|\ge6$. 
    Since $\delta(H'_{3-j})\ge1$, let $e\in\partial_{H'_{3-j}}(x_{3-j})$.
    Then $e\in\partial_{G'_j}(x_j)$.
    Let $H'_j$ be the underlying simple graph of $G'_j$ that contains $e$.  
    Then $e\in E(H'_1)\cap E(H'_2)$. 
    Since multiple edges are removable, we have $E_N(G'_j)\subseteq E(H'_j)$.
    Moreover, $H'_j$ is either a 4-cycle or $K_4$, so $\delta(H'_j)\ge2$. 
    If $j=1$, i.e., $|V(G'_1)|=4$, then $H'_1$ is a 4-cycle since $G'_1$ has a nontrivial barrier.  
    Hence $|E(H'_1)\cup E(H'_2)|=|E(H'_1)|+|E(H'_2)|-|E(H'_1)\cap E(H'_2)|\le4+(5|V(G'_2)|-10)/2-1=(5|V(G)|-14)/2$. 
    If $j=2$, i.e., $|V(G'_2)|=4$, then $|E(H'_2)|\le6$, and hence $|E(H'_1)\cup E(H'_2)|\le(5|V(G'_1)|-14)/2+6-1=(5|V(G)|-14)/2$. 
    Therefore, there exist spanning subgraphs $H'_1$ of $G'_1$ and $H'_2$ of $G'_2$ such that $E_N(G'_i)\subseteq E(H'_i)$ and $\delta(H'_i)\ge1$ for each $i\in\{1,2\}$, and $|E(H'_1)\cup E(H'_2)|\le(5|V(G)|-14)/2$.

    Let $H$ be the spanning subgraph of $G$ with edge set $E(H'_1)\cup E(H'_2)$. 
    Clearly, $V(H)=V(G)$, $\delta(H)\ge1$.
    By Corollary \ref{cor:NR_TC}, $E_N(G)=E_N(G'_1)\cup E_N(G'_2)\subseteq E(H'_1)\cup E(H'_2)= E(H)$. 
    Moreover, $|E(H)|=|E(H'_1)\cup E(H'_2)|\le(5|V(G)|-14)/2$. 
    Therefore, the lemma follows. 
\end{proof}

\vspace{1em}

\noindent{\em Proof of Theorem \ref{thm:main-thm}.} 
Since $G$ is a minimal matching covered graph, we have $E(G)=E_N(G)$. 
Thus, $|E(G)|=|E_N(G)|\le (5|V(G)|-10)/2$ by Lemma \ref{lem:H_in_MC}.

 Now we will show that for every integer $k\ge2$, ${G}_k$ is a minimal matching covered graph with $|E({G}_k)|=(5|V({G}_k)|-10)/2$.  
Note that for any integer $k\ge2$, ${G}_k$ is connected. 
For each $i\in\{1,\ldots,k\}$, both
$\{au_i,bv_i\}\cup\{u_jv_j:j\ne i\}$ and
$\{av_i,bu_i\}\cup\{u_jv_j:j\ne i\}$ are perfect matchings of ${G}_k$.  
Since every edge of ${G}_k$ lies in one of these perfect matchings, ${G}_k$ is matching covered.  
Moreover, since $au_2$ is not allowed in ${G}_k-u_1v_1$ and $bv_1$ is not allowed in ${G}_k-au_1$, $u_1v_1$ and $au_1$ are not removable in ${G}_k$. 
By symmetry, every edge of ${G}_k$ is nonremovable.
Therefore, ${G}_k$ is minimal matching covered. 
Note that $d_{{G}_k}(a)=2k=d_{{G}_k}(b)$ and for each $1\le i\le k$,  $d_{{G}_k}(u_i)=3=d_{{G}_k}(v_i)$.
Since $|V({G}_k)|=2k+2$, we have $|E({G}_k)|=5k=(5|V({G}_k)|-10)/2$.
Therefore, the family $\{{G}_k:k\ge2\}$ of graphs attains the upper bound in this theorem. 

 Conversely, we assume that $G$ is a minimal matching covered graph with at least 6 vertices and $|E(G)|=(5|V(G)|-10)/2$. 
By Lemma \ref{lem:H_in_MC}, $G$ is bicritical. 
By Proposition \ref{pro:minimal_MC_bicritical},   $G$ is a minimal bicritical graph. 
By Lemma \ref{lem:minimal_bicritical}, $G$ is either a wheel of order 6 or $G\in\{{G}_k:k\ge2\}$. 
It can be checked that the wheel of order 6 is not a minimal matching covered graph. Thus, $G\in\{{G}_k:k\ge2\}$. 
Therefore, the theorem follows. 
$\hfill\qedsymbol$
\vspace{1em} 
 
\section{Proof of Theorem \ref{thm:nonbicritical}}
Let $G$ be a matching covered nonbipartite graph with a nontrivial barrier $B$. 
Let $H(G,B)$ denote the graph obtained from $G$ by contracting every nontrivial odd component of $G-B$ to a singleton. 
Since $B$ is an independent set and $G-B$ has no even components by Corollary \ref{cor:M-C-without-even-components}, $H(G,B)$ is a bipartite graph with $B$ as one partite set. 
Since for any nontrivial odd component $Q$ of $G-B$, $\partial_G(V(Q))$ is a nontrivial barrier cut of $G$, using Lemma \ref{lem:C-contractions-MC} repeatedly,  $H(G,B)$ is matching covered. 
When no confusion arises, we assign the same label to any vertex (or edge) common to both $G$ and  $H(G,B)$.
Note that a vertex of $H(G,B)$ but not in $G$ must be gotten by contracting a nontrivial odd component of $G-B$.

A graph $G$ is {\em factor-critical} if for any $v\in V(G)$, $G-v$ has a perfect matching. Obviously, every bipartite graph with at least 3 vertices is not factor-critical. The following lemma is taken from Exercise 3.3.18 (b) of \cite{LP86} or Theorem 3.2 of \cite{Lucchesi2024}. 
\begin{lem}[\cite{LP86},\cite{Lucchesi2024}]\label{lem:max-barrier-critical}
    Let $B$ be a maximal barrier in a graph $G$ that has a perfect matching. Then each component of $G-B$ is a factor-critical graph.
\end{lem}

For every matching covered nonbipartite graph $G$ with a maximal barrier $B$, $G-B$ has a nontrivial odd component by Lemma \ref{lem:max-barrier-critical}. 

\begin{lem}\label{lem:edge-count}
Let $G$ be a minimal matching covered nonbipartite graph with a maximal nontrivial barrier $B$. 
Let $Q_1,\ldots,Q_r$ be the nontrivial components of $G-B$, $r\geq 1$, and let $G'_i:=G/\overline{V(Q_i)}$ for $1\le i\le r$. 
Let $\mathcal{K}_4:=\{G'_i:1\le i\le r,\ |V(G'_i)|=4\}$. 
Then \\ 
{\rm (i)} every graph of $\mathcal{K}_4$ has $K_4$ as its underlying simple graph, and \\ 
{\rm (ii)} 
$E(G)=\cup_{i=1}^{r}E_N(G'_i)\cup E_N(H(G,B))$, and \\  
{\rm (iii)}
$\sum_{i=1}^{r}|V(G'_i)|=|V(G)|-2|B|+2r$, and \\ 
{\rm (iv)} 
$|E(G)|\le\frac{5|V(G)|}{2}-5|B|+|\mathcal{K}_4|+|E_N(H(G,B))|$. 
\end{lem}
\begin{proof}
(i) By Corollary \ref{cor:M-C-without-even-components}, $B$ is independent and $G-B$ has exactly $|B|$ components, all of which are odd. 
For each $1\le i\le r$, $\partial_G(V(Q_i))$ is a
nontrivial barrier cut, so $G'_i$ is matching covered by Lemma \ref{lem:C-contractions-MC}. 
By Lemma \ref{lem:max-barrier-critical}, each $Q_i$ is factor-critical and hence nonbipartite.
Thus $G'_i$ is nonbipartite. 
Note that the underlying simple graph of a matching covered graph of order 4 is either a 4-cycle or $K_4$. 
Therefore, every graph in $\mathcal{K}_4$ has $K_4$ as its underlying simple graph and so, (i) holds. 

(ii) Since $G$ is minimal matching covered, we have $E(G)=E_N(G)$.  
For each $i\in \{1,\ldots,r-1\}$, $B$ is also a barrier of $G/V(Q_1)/\ldots/V(Q_i)$, and $\partial(V(Q_j))$ is also a nontrivial barrier cut for each $j\in\{i+1,\ldots,r\}$.  
Using Corollary \ref{cor:NR_TC} repeatedly, we have $E(G)=\cup_{i=1}^{r}E_N(G'_i)\cup E_N(H(G,B))$ and hence, (ii) holds. 

(iii)  
Note that $|V(H(G,B))|=2|B|$. 
Using Fact \ref{fact:C-contractions-vtx-number}
repeatedly, we have
\begin{equation}\label{eq:count_VGi}
\begin{aligned}
\sum_{i=1}^{r}|V(G'_i)|
&=|V(G)|-|V(H(G,B))|+2r\\
&=|V(G)|-2|B|+2r.
\end{aligned}
\end{equation} 
Thus, (iii) holds. 

(iv) 
For every $G'_j\in\mathcal{K}_4$, $G'_j$ has $K_4$ as its underlying simple graph by (i). 
Since multiple edges are removable, we have $|E_N(G'_j)|\le6$.  
For every $G'_l\in\{G'_1,\ldots, G'_r\}\setminus \mathcal{K}_4$, $|V(G'_l)|\ge6$ and so, $|E_N(G'_l)|\le\frac{5|V(G'_l)|-10}{2}$ by Lemma \ref{lem:H_in_MC}. 
Thus, by \eqref{eq:count_VGi}, we have 
\begin{equation}\label{eq:count_ENGi}
\begin{aligned}
\sum_{i=1}^{r}|E_N(G'_i)|
&\le
\frac{5}{2}\sum_{i=1}^{r}(|V(G'_i)|-2)
+|\mathcal{K}_4|\\
&=\frac{5}{2}\sum_{i=1}^{r}|V(G'_i)|
-5r+|\mathcal{K}_4|\\
&=\frac{5|V(G)|}{2}-5|B|+|\mathcal{K}_4|,
\end{aligned}
\end{equation}
where the last equality follows from (iii). And  (iv) follows from (ii) and \eqref{eq:count_ENGi}.
\end{proof}

\begin{lem}\label{lem:bicritical_with_1RE}
Let $G$ be a bicritical graph with $|V(G)|\ge6$. 
If $G$ has exactly one removable edge $e$ and $|E(G)|=\frac{5|V(G)|-8}{2}$, then $G\cong {G}_k+ab$ for some $k\ge2$, and $e=ab$.
\end{lem}
\begin{proof}
    Since multiple edges are removable, $G$ is simple. 
    Assume first that $|V(G)|=6$. 
    Then $|E(G)|=11$ and so, $\Delta(G)\ge4$. 
    If $G$ is 3-connected, then it is a brick distinct from $K_4$ and $\overline{C_6}$. 
    By Lemma \ref{lem:re_in_brick}, $G$ has at least $\Delta(G)-2\ge2$ removable edges, a contradiction. 
    By Corollary \ref{cor:MC_2-con}, $G$ is 2-connected. 
    Now we assume that $G$ has a 2-vertex cut $\{a,b\}$. 
    As $G$ is bicritical, $G-\{a,b\}$ has a perfect matching. 
    Thus, the components of $G-\{a,b\}$ are two copies of $K_2$, say $u_1v_1$ and $u_2v_2$.  
    Moreover, every vertex of $G$ has at least 3 neighbors and hence each of $u_1,v_1,u_2,v_2$ is adjacent to both $a$ and $b$. Together with $|E(G)|=11$, we have $G\cong {G}_2+ab$. 
    As $G-ab\cong {G}_2$ is matching covered, the unique removable edge $e=ab$. 

    Now assume that $|V(G)|\ge8$. 
    Let $H$ be a minimal bicritical spanning subgraph of $G$. 
    By the same argument in the first paragraph of proof of Lemma \ref{lem:H_in_MC}, it can be shown that $E_N(G)\subseteq E(H)$.  
    Thus, $E(G)\setminus E(H)\subseteq\{e\}$.  
    By Lemma \ref{lem:minimal_bicritical}, $|E(H)|\le\frac{5|V(H)|-10}{2}=|E(G)|-1$. 
    Therefore, $E(G)\setminus E(H)=\{e\}$, i.e., $H=G-e$, and $|E(H)|=\frac{5|V(H)|-10}{2}$.  
    By Lemma \ref{lem:minimal_bicritical} again, we may assume that $H={G}_k$ for some $k\ge3$ (as $|V(H)|=|V(G)|\ge8$).  
    Suppose to the contrary that $e\neq ab$. 
    Since $G$ is simple, by relabeling the vertices if necessary, we may assume that $e=u_1u_2$, i.e., $G={G}_k+u_1u_2$. 
    The edge sets $\{av_i,bu_i\}\cup\{u_jv_j:j\neq i\}$ for each $i\in\{1,\ldots,k\}$, $\{au_s,bv_s\}\cup\{u_tv_t:1\le t\le k,\ t\neq s\}$ for each $s\in\{2,\ldots,k\}$, and $\{u_1u_2,bv_1,av_2\}\cup\{u_jv_j:3\le j\le k\}$ are perfect matchings of $G-au_1$. Since every edge of $G-au_1$ lies in one of these perfect matchings and $G-au_1$ is connected, $G-au_1$ is matching covered. Hence $au_1$ is a removable edge of $G$ distinct from $e$, a contradiction.  
    Therefore, $e=ab$ and $G\cong {G}_k+ab$.
\end{proof}

\begin{lem}\label{lem:Hk}
    Every graph $G\in \{{H}_k:k\ge2\}$ is a minimal matching covered nonbipartite graph that is not bicritical and $|E(G)|=(5|V(G)|-14)/2$. 
\end{lem}
\begin{proof}
    Let $G={H}_k={G}_k-av_1-bu_1$ for some $k\ge2$. 
    Since $|E(G_k)|=(5|V(G_k)|-10)/2$ by Theorem \ref{thm:main-thm}, we have $|E(G)|=|E(G_k)|-2=(5|V(G)|-14)/2$. 
     
    For each $i\in\{1,\ldots,k\}$ and each $t\in\{2,\ldots,k\}$, $\{au_i,bv_i\}\cup\{u_jv_j:j\ne i\}$ and  $\{av_t,bu_t\}\cup\{u_jv_j:j\ne t\}$ are perfect matchings of $G$. 
    Since $G$ is connected and every edge of $G$ lies in one of these perfect matchings, $G$ is matching covered. 

    Since $au_2$ is not allowed in $G-u_1v_1$  and $au_1$ is not allowed in $G-u_2v_2$, we have $u_1v_1,u_2v_2\in E_N(G)$. 
    Similarly, since $bv_1$ is not allowed in $G-au_1$ and $bv_2$ is not allowed in $G-au_2$, we have $au_1,au_2\in E_N(G)$.  
    By symmetry, every edge of $G$ is nonremovable and so, $G$ is minimal matching covered. 

    Since $au_2v_2a$ is a triangle of $G$, $G$ is nonbipartite. 
    Since $N_G(u_1)=\{a,v_1\}$, $G-\{a,v_1\}$ has $u_1$ as a trivial component and hence no perfect matchings. Thus, $G$ is not bicritical. Therefore, the lemma holds.  
\end{proof}
\vspace{1em}
\noindent{\em Proof of Theorem \ref{thm:nonbicritical}.}  
    Since $G$ is a minimal matching covered nonbipartite graph and $G$ is not bicritical,  we have $E(G)=E_N(G)$ and hence $|E(G)|\le\frac{5|V(G)|-14}{2}$ by Lemma \ref{lem:H_in_MC}. Next we mainly show that if the equality holds, then $G$ is isomorphic to some $H_k$. 

    Assume that $|E(G)|=\frac{5|V(G)|-14}{2}$. 
    As usual we know that  $G$ is  simple and 2-connected. 
    By Proposition \ref{pro:mc_is_Bi}, $G$ has a nontrivial barrier. 
    Let $B$ be a maximal nontrivial barrier of $G$, and let $Q_1,\ldots,Q_r$ be the nontrivial components of $G-B$.
    For each $i\in\{1,\ldots,r\}$, let $G'_i:=G/(\overline{V(Q_i)}\to\overline{q_i})$ and let $\mathcal{K}_4:=\{G'_i:1\le i\le r, |V(G'_i)|=4\}$. 
    Since $G$ is nonbipartite, $r\ge1$. 
    Let $H$ be a minimal matching covered spanning subgraph of $H(G,B)$. 
    Since $B$ is nontrivial, we have $|B|\ge2$ and hence $|V(H)|=|V(H(G,B))|=2|B|\ge4$. Moreover, we have the following claim. 
    \begin{cla}
        $E_N(H(G,B))\subseteq E(H)$.
    \end{cla} 
    \begin{proof} 
    For any $e\in E(H(G,B))\setminus E(H)$, we will show that $e$ is removable in $H(G,B)$. 
    Since $H$ is matching covered and bipartite,   
    for every $u\in B$ and $v\in V(H(G,B))\setminus B$, $H-\{u,v\}$ has a perfect matching $M$ by Lemma \ref{lem:char_MC_bi}.   
    Moreover, $M$ is a perfect matching of $(H(G,B)-e)-\{u,v\}$. 
    By Lemma \ref{lem:char_MC_bi} again, $H(G,B)-e$ is matching covered.    
    Thus, $e$ is removable in $H(G,B)$, and hence $E_N(H(G,B))\subseteq E(H)$. 
    \end{proof} 

    By Lemma \ref{lem:edge-count} (ii) and Claim 1, we have $E(G)=\cup_{i=1}^{r}E_N(G'_i)\cup E_N(H(G,B))\subseteq \cup_{i=1}^{r}E_N(G'_i)\cup E(H)
    \subseteq E(G)$. 
    Thus, $E(G)=\cup_{i=1}^{r}E_N(G'_i)\cup E(H)$.  
    Since $E(G'_1),\ldots, E(G'_r)$ are pairwise disjoint, we have 
    \begin{equation}\label{eq:nonbicritical_EG}
        |E(G)|=|E(H)|+\sum_{i=1}^{r}|E_N(G'_i)\setminus E(H)|.
    \end{equation}

    For every $G'_i\notin\mathcal{K}_4$, $|V(G'_i)|\ge6$ and hence $|E_N(G'_i)|\le(5|V(G'_i)|-10)/2$ by Lemma \ref{lem:H_in_MC}. 
    For every $G'_i\in\mathcal{K}_4$, $G'_i$ has $K_4$ as its underlying simple graph by Lemma \ref{lem:edge-count} (i) and so, $|E_N(G'_i)|\le6$. If  the equality holds, then  $G'_i\cong K_4$.    
    So $E_N(G'_i)=E(G'_i)$.
    Since $H$ is 2-connected (by Corollary \ref{cor:MC_2-con}), we have $\delta(H)\ge2$. 
    Then $|\partial_G(V(Q_i))\cap E(H)|\ge2$ and hence $|E_N(G'_i)\setminus E(H)|\le6-2=4$. 
    Therefore, in both cases, 
    \begin{equation}\label{eq:nonbicritical_ENGi}
        |E_N(G'_i)\setminus E(H)|\le\frac{5|V(G'_i)|-10}{2},~\mbox{for each}~1\le i\le r.
    \end{equation} 

    If $|B|\ge3$, then $|V(H)|=2|B|\ge6$ and hence, $|E(H)|\le (3|V(H)|-6)/2=3|B|-3$ by Theorem \ref{thm:minimal-bi-MC}. 
    Thus, by \eqref{eq:nonbicritical_EG}, \eqref{eq:nonbicritical_ENGi} and Lemma \ref{lem:edge-count} (iii), we have 
    \setlength{\jot}{0pt} 
    \begin{equation*}
        \begin{aligned}
            |E(G)|
            &=|E(H)|+\sum_{i=1}^{r}|E_N(G'_i)\setminus E(H)|\\
            &\le3|B|-3+ \frac{5}{2}\sum_{i=1}^{r}|V(G'_i)|-5r ~~(\mbox{by \eqref{eq:nonbicritical_ENGi}})\\
            &=\frac{5|V(G)|}{2}-2|B|-3~~(\mbox{by  Lemma \ref{lem:edge-count} (iii)})\\
            &\le\frac{5|V(G)|-18}{2}.
        \end{aligned}
    \end{equation*} 
    It contradicts the assumption that $|E(G)|=(5|V(G)|-14)/2$.   
    Therefore, $|B|=2$. 

    Let $B=\{s_1,s_2\}$. 
    Since $H(G,B)$ has exactly 4 vertices, $H(G,B)$ has a 4-cycle as its underlying simple graph.  
    Since $E(G'_1)$, $\ldots$, $E(G'_r)$ are pairwise disjoint, for each $i\in\{1,\ldots,r\}$, by Lemma \ref{lem:edge-count} (ii), we have 
    \begin{equation}\label{eq:RE}
        E(G'_i)\setminus E_N(G'_i)\subseteq E_N(H(G,B))\cap\partial_G(V(Q_i)).
    \end{equation}
    In other words, every removable edge of $G'_i$ is incident with $\overline{q_i}$ and nonremovable in $H(G,B)$. 

    \begin{cla}
        For each $i\in\{1,\ldots,r\}$, $G'_i$ has at most one removable edge. 
    \end{cla}
    \begin{proof}
    Suppose that $G'_i$ has two removable edges $f$ and $f'$ for some $i\in\{1,\ldots,r\}$. 
    By \eqref{eq:RE}, $\{f,f'\}\subseteq E_N(H(G,B))\cap\partial_G(V(Q_i))$. 
    Let $q_i\in V(H(G,B))$ be the vertex obtained by contracting $V(Q_i)$.  
    Since $H(G,B)$ has a 4-cycle as its underlying simple graph and multiple edges are removable, we have $\partial_{H(G,B)}(q_i)=\{f,f'\}$. 
    Hence $\partial_G(V(Q_i))=\{f, f'\}$.  
    Then $G'_i-f$ has a cut edge $f'$ and hence is not 2-connected. 
    Since $Q_i$ is nontrivial, we have $|V(G'_i)|\ge4$ and hence $G'_i-f$ is not matching covered by Corollary \ref{cor:MC_2-con}. 
    It contradicts the assumption that $f$ is removable in $G'_i$. Hence the claim holds. 
    \end{proof} 
    For each $i\in\{1,\ldots,r\}$, $G'_i$ is simple by Claim 2. 
    If $|V(G'_i)|=4$, then $G'_i\cong K_4$ by Lemma \ref{lem:edge-count} (i) and Claim 2 and hence $|E(G'_i)|=6$.  
    If $|V(G'_i)|\ge6$, then $|E(G'_i)|\le |E_N(G'_i)|+1\le\frac{5|V(G'_i)|-8}{2}$ by Claim 2 and Lemma \ref{lem:H_in_MC}. 
    Therefore, $|E(G'_i)|\le\frac{5|V(G'_i)|-8}{2}$ for each $i\in\{1,\ldots,r\}$. 

    If both components of $G-B$ are nontrivial, then $r=2$, $|V(G'_1)|+|V(G'_2)|=|V(G)|$ and $|E(G)|=|E(G'_1)|+|E(G'_2)|$. 
    Hence, $|E(G)|\le\frac{5|V(G)|-16}{2}$, contradicting the assumption that $|E(G)|=(5|V(G)|-14)/2$.  
    So $G-B$ consists of $Q_1$ and a trivial component $z$. 
    Since $G$ is simple and 2-connected, we have $N_G(z)=\{s_1,s_2\}$. 
    Then $|V(G'_1)|=|V(G)|-2$ and 
    \begin{equation}\label{eq:nonbicritical_EG1}
        |E(G'_1)|=|E(G)|-2=\frac{5|V(G'_1)|-8}{2}.
    \end{equation}

    If $|V(G)|=6$, then $G'_1\cong K_4$ and $|\partial_G(V(Q_1))|=3$. 
    Since $G$ is 2-connected, one in  $s_1,s_2$ has one neighbor in $Q_1$ and the other one has two neighbors in $Q_1$. 
    Therefore, $G$ is obtained from $K_4$ by replacing one edge with a path of length 3 and so, $G\cong {H}_2$. 

    Now assume that $|V(G)|\ge8$. 
    Then $|V(G'_1)|\ge6$. 
    By Lemma \ref{lem:H_in_MC},  $|E_N(G'_1)|\le\frac{5|V(G'_1)|-10}{2}$.  So Eq. \eqref{eq:nonbicritical_EG1} implies $|E(G'_1)|-|E_N(G'_1)|\ge1$.  
    Together with Claim 2, we have $|E(G'_1)|-|E_N(G'_1)|=1$.  
    By \eqref{eq:nonbicritical_EG1} and Lemma \ref{lem:H_in_MC}, $G'_1$ is bicritical.    
    By Lemma \ref{lem:bicritical_with_1RE}, $G'_1\cong {G}_k+ab$ for some $k\ge2$, and $e=ab$. 
    By \eqref{eq:RE}, $e\in E_N(H(G,B))\cap \partial_G(V(Q_1))$. Without loss of generality assume that $e=s_2v$ in $G$, where $v\in V(Q_1)$.  
    Since multiple edges are removable, we have $E_G[\{s_2\},V(Q_1)]=\{s_2v\}$. 
    Since $B$ is independent and $N_G(z)=B$, we have $N_G(s_2)=\{z,v\}$. 
    Let $\overline{q_1}=s_1$. 
    Then $G$ is obtained from $G'_1$ by replacing $e$ with the path $s_1zs_2v$. 
    Therefore, $G\cong {H}_{k+1}$.

    Conversely, let $G={H}_k$ for some $k\ge2$. 
    By Lemma \ref{lem:Hk}, $|E(G)|=(5|V(G)|-14)/2$ and hence, the theorem follows.  $\hfill\qedsymbol$

\section{Proof of Theorem \ref{thm:main2}}

For any integer $k\ge2$, each $G\in \mathcal{G}_k$ can be obtained by  splicing operations. 
Let $K_{3,3}^+$ be obtained from $K_{3,3}$ with partite sets $\{s_1,s_2,s_3\}$ and $\{x_1,x_2,x_3\}$ by adding $2k-3$ edges, each parallel to one of $s_1x_3$, $s_2x_3$ and $s_3x_3$. For  a bijection $\theta:\partial_{K_{3,3}^+}(x_3)\to \partial_{{G}_k}(b)$ such that $\theta^{-1}(bu_i)$ and $\theta^{-1}(bv_i)$ are not parallel in $K_{3,3}^+$ for each $1\le i\le k$, $G$ is obtained by splicing $K_{3,3}^+$ at $x_1,x_2$ and $x_3$ with $K_4$, $K_4$ and $G_k$ at $b$ successively.

\begin{lem}\label{lem:mathcalG}
    Every graph $G\in \cup_{k\ge2} \mathcal{G}_k$ is a 3-connected minimal matching covered graph and $|E(G)|=(5|V(G)|-26)/2$. 
\end{lem}
\begin{proof} 
Let $G\in \mathcal{G}_k$ for some integer $k\ge2$.  
Since  $K_{3,3}^+$,  $K_4$ and ${G}_k$ are also matching covered, by the spliced construction of $G$ and  Proposition \ref{pro:MC_IS_MC} repeatedly, $G$ is matching covered.  
Then $G$ is 2-connected by Corollary \ref{cor:MC_2-con}. 

It can be checked that  $|V(G)|=|V({G}_k)|+8$ and  $|E(G)|=|E({G}_k)|+12$. 
Since $|E({G}_k)|=(5|V({G}_k)|-10)/2$ by Theorem \ref{thm:main-thm}, we have $|E(G)|=\frac{5|V({G}_k)|-10}{2}+12=\frac{5|V(G)|-26}{2}$.

Let $B:=\{s_1,s_2,s_3\}$. Then
 $G-B$ has exactly 3 components $Q_1$, $Q_2$ (two copies of a triangle) and $Q:=Q_3={G}_k-b$.  
Thus, $B$ is a barrier of $G$.
 
\begin{claim}
 $G$ is 3-connected.
\end{claim}
\begin{proof} 
Suppose to the contrary that $G$ has a 2-vertex cut $S$. 
Note that $\overline{V(Q)}=B\cup V(T_1)\cup V(T_2)$. 
It can be checked that $G[\overline{V(Q)}]$ is 2-connected, as $G[\overline{V(Q)}]$ consists of two disjoint triangles joined by the three vertex-disjoint paths $t_1s_1t'_1$, $t_2s_2t'_2$ and $t_3s_3t'_3$. 

If $S\subseteq \overline{V(Q)}$, let $Q_a$ be the component of $G-S$ that contains $a$. 
Since $Q$ is connected and $N_G(s_j)\cap V(Q)\neq\emptyset$ for each $j\in\{1,2,3\}$, we have
$V(Q)\cup(B\setminus S)\subseteq V(Q_a)$.
For each $i\in\{1,2\}$, as $\partial_G(V(T_i))$ is a matching of size $3$, we have $|\partial_G(V(T_i))\cap E(G-S)|\ge3-|S|=1$. 
Since $T_i-S$ is nonempty and connected, $V(T_i)\setminus S\subseteq V(Q_a)$.
Thus, $V(G-S)\subseteq V(Q_a)$, a contradiction.

Now we may assume that $S\cap \overline{V(Q)}=\{q_1\}$ and $S\cap V(Q)=\{q_2\}$. 
Since $G[\overline{V(Q)}]$ is 2-connected, $G[\overline{V(Q)}]-q_1$ is connected.   
If $q_2=a$, then  at least one of $u_i,v_i$ has a neighbor in  $B\setminus \{q_1\}$, so $G-S$ is connected. 
If $q_2\neq a$, then $Q-q_2$ is connected and there is $1\leq i\leq k$ such that $u_i$ and $v_i$ are different from $q_2$. As above, at least one of $u_i,v_i$ has a neighbor in $B\setminus \{q_1\}$. Thus, $G-S$ is connected, a contradiction.

Finally assume that  $S\subseteq V(Q)$.  We can see that each vertex in $V(Q)\setminus S$ except $a$ has a neighbor in $B$. Hence $G-S$ is connected, a contradiction.  
\end{proof}

Since $K_4$ and ${G}_k$ are minimal matching covered, we have $E(G/\overline{V(Q_j)})=E_N(G/\overline{V(Q_j)})$ for each $j\in\{1,2,3\}$. 
Moreover, $\partial(V(Q_j))$ is a nontrivial barrier cut of $G$ for each $j\in\{1,2,3\}$. 
By Corollary \ref{cor:NR_TC}, $E_N(G'_j)\subseteq E_N(G)$ for each $j\in\{1,2,3\}$.  
Since $B$ is independent, $E(G)=\cup_{j=1}^{3}E(G/\overline{V(Q_j)})=\cup_{j=1}^{3}E_N(G/\overline{V(Q_j)})\subseteq E_N(G)$. 
Thus, $E(G)=E_N(G)$ and hence $G$ is minimal matching covered. 
 \end{proof}
\vspace{1em}

\noindent{\em Proof of Theorem \ref{thm:main2}.} 
    If $G$ is bicritical, then $G$ is a brick.  Since $G$ has no removable edges, $G$ is isomorphic to $K_4$ or $\overline{C_6}$ by Lemma \ref{lem:re_in_brick},  contradicting the assumption that $|V(G)|\ge14$.  
    So $G$ is not bicritical. 
    By Proposition \ref{pro:mc_is_Bi}, $G$ has a nontrivial barrier. 

    Let $B$ be a maximal nontrivial barrier of $G$.
    Since $G$ is 3-connected, we have $|B|\ge3$. 
    By Corollary \ref{cor:M-C-without-even-components}, $B$ is an independent set and $G-B$ has no even components. Then $G-B$ has a nontrivial odd component since $G$ is nonbipartite. 
    Let $Q_1,\ldots, Q_r$ ($r\ge1$) denote all nontrivial odd components of $G-B$ such that $|V(Q_i)|\le|V(Q_{i+1})|$ for $1\le i\le r-1$. 
    Let $G'_i:=G/(\overline{V(Q_i)}\to\overline{q_i})$ for each $1\le i\le r$ and $\mathcal{K}_4:=\{G'_i:1\le i\le r, |V(G'_i)|=4\}$.  Then each $G'_i$ is a matching covered graph by Lemma \ref{lem:C-contractions-MC}. 
    By Lemma \ref{lem:edge-count} (iv), we have 
    \begin{equation}\label{eq:3-con}
        |E(G)|\le\frac{5|V(G)|}{2}-5|B|+|\mathcal{K}_4|+|E_N(H(G,B))|.
    \end{equation}

    Since $H(G,B)$ is bipartite, any tight cut decomposition of $H(G,B)$ produces only braces.
    Let $H'_1,\ldots, H'_t$ ($t\ge1$) denote all braces obtained from $H(G,B)$ by the tight cut decomposition, such that $|V(H'_i)|\le|V(H'_{i+1})|$ for each $1\le i\le t-1$.
    By Corollary \ref{cor:NR_TC}, $E_N(H(G,B))=E_N(H'_1)\cup\ldots\cup E_N(H'_t)$. 
    Note that $|V(H'_i)|\ge4$ for each $1\le i\le t$. 
    Let $\mathcal{C}_4:=\{H'_i:1\le i\le t, |V(H'_i)|=4\}$.   
    For every $j\ge |\mathcal{C}_4|+1$, $E_N(H'_j)=\emptyset$ by Lemma \ref{lem:RE_brace}.  
    For every $H'_i\in \mathcal{C}_4$, the underlying simple graph of $H'_i$ is a 4-cycle.
    Since every edge cut of $H'_i$ is also an edge cut of $G$ and $G$ is 3-connected, every edge cut of $H'_i$ contains at least 3 edges.  
    Since any two edges of a 4-cycle form an edge cut,  at most one edge of $H'_i$ is not a multiple edge.  
    Since multiple edges are removable, $H'_i$ contains at most one nonremovable edge. 
    Thus, 
    \begin{equation}\label{eq:3}
        |E_N(H(G,B))|\le|\mathcal{C}_4|.
    \end{equation}
    Note that $|\mathcal{K}_4|\le o(G-B)=|B|$. Then 
    \begin{equation}\label{eq:4}
        |\mathcal{K}_4|\le |B|.
    \end{equation}   
 
    Using Fact \ref{fact:C-contractions-vtx-number} repeatedly, we have $|V(H(G,B))|=\sum_{i=1}^t|V(H'_i)|-2(t-1)$.  Moreover, since $|V(H'_i)|\ge4$ for each $1\le i\le t$, we have $|V(H(G,B))|=\sum_{i=1}^t|V(H'_i)|-2(t-1)\ge 4t-2(t-1)=2t+2$ 
    and hence $|\mathcal{C}_4|\le t\le \frac{|V(H(G,B))|}{2}-1=|B|-1$. 
    Then $|E_N(H(G,B))|\le|\mathcal{C}_4|\le t$ by \eqref{eq:3}. Thus, we have  
    \begin{equation}\label{eq:5}
        |E_N(H(G,B))|\le |B|-1. 
    \end{equation}

    Moreover, the first part of the theorem follows from  the following claim. 
    \begin{claim}
         $|E(G)|\le\frac{5|V(G)|-26}{2}$, with equality only if $|B|=3$ and $|\mathcal{K}_4|=2$. Moreover, if $|B|=3$, then $E_N(H(G,B))=\emptyset$. 
    \end{claim}
    \begin{proof}
    If $|B|\ge5$, then combining \eqref{eq:3-con}, \eqref{eq:4} and \eqref{eq:5}, we have $|E(G)|\le\frac{5|V(G)|}{2}-3|B|-1\le\frac{5|V(G)|-32}{2}$. 
    
    Now we assume that $|B|=4$.  Then $H(G,B)$ has exactly 8 vertices.  
    If $H(G,B)$ is a brace, then $E_N(H(G,B))=\emptyset$ by Lemma \ref{lem:RE_brace} and hence, by \eqref{eq:3-con} and \eqref{eq:4}, $|E(G)|\le \frac{5|V(G)|}{2}-16$. 
    If $H(G,B)$ is not a brace, then $H(G,B)$ has a nontrivial tight cut $\partial(X)$ such that $|X|=3$ and $|\overline{X}|=5$. 
    Let $I:=V(H(G,B))\setminus B$. 
    Since $G$ is 3-connected, every vertex of $I$ has at least three distinct neighbors in $B$. By Proposition \ref{pro:tight-cut-in-bipartiteMC}, $||X\cap I|-|X\cap B||=1$ and so, $|X\cap I|=1$ or 2. 
    If $|X\cap I|=1$, then by Proposition \ref{pro:tight-cut-in-bipartiteMC}, no edge of $H(G,B)$ joins $X\cap I$ to $\overline{X}$ and hence the unique vertex of $X\cap I$ has at most two neighbors, a contradiction. 
    So $|X\cap I|=2$ and $|X\cap B|=1$. 

    Let $H':=H(G,B)/(X\to x)$ and $H'':=H(G,B)/(\overline{X}\to \overline{x})$.
    Then $|V(H')|=6$ and $|V(H'')|=4$. 
    Let $B':=B\setminus X$ and $I':=(I\setminus X)\cup\{x\}$. 
    Then $B'$ and $I'$ are the partite sets of $H'$. 
    By Proposition \ref{pro:tight-cut-in-bipartiteMC} again, no edge of $H(G,B)$ joins $X\cap B$ to $I'\setminus\{x\}$. 
    Hence every vertex of $I'\setminus\{x\}$ is adjacent to all three vertices of $B'$. 
    Let $G[Y_1]$, \ldots, $G[Y_4]$ denote the components of $G-B$, and let $y_i$ be the vertex obtained by contracting $Y_i$. 
    Adjust notation so that $X\cap I=\{y_1,y_2\}$. 
    Then in $H(G,B)$, we have $\{y_1,y_2\}\subset X$ and $\{y_3,y_4\}\subset \overline{X}$. 
    Since $E_{H(G,B)}[X\cap B, \overline{X}\cap I]=\emptyset$ by Proposition \ref{pro:tight-cut-in-bipartiteMC}, in $G$, we have $E_G[X\cap B, Y_3\cup Y_4]=\emptyset$. 
    Since $B$ is independent,  $G-N_{H'}(x)$ is not connected. 
    Since $G$ is 3-connected, $|N_{H'}(x)|\ge3$. 
    Thus, $H'$ has $K_{3,3}$ as its underlying simple graph and is therefore a brace. 
    By Lemma \ref{lem:RE_brace}, $E_N(H')=\emptyset$. 
    Since the underlying simple graph of $H''$ is a 4-cycle, by the same argument for graphs in $\mathcal{C}_4$, we have $|E_N(H'')|\le1$. 
    Then by Corollary \ref{cor:NR_TC}, $|E_N(H(G,B))|\le |E_N(H')|+|E_N(H'')|\le1$. 
    Thus, by \eqref{eq:3-con} and \eqref{eq:4}, we have $|E(G)|\le \frac{5|V(G)|}{2}-4|B|+1\le \frac{5|V(G)|}{2}-15$. 
    In short, if $|B|=4$, then $|E(G)|\le \frac{5|V(G)|-30}{2}$.

    Finally, we consider the remaining case: $|B|=3$. 
    Then $|V(H(G,B))|=6$.
    Since $G$ is 3-connected, we have $N_{H(G,B)}(u)=B$ for every $u\in V(H(G,B))\setminus B$. 
    Thus, $H(G,B)$ has $K_{3,3}$ as its underlying simple graph and is therefore a brace. 
    By Lemma \ref{lem:RE_brace}, $E_N(H(G,B))=\emptyset$.   
    Note that $|\mathcal{K}_4|\le |B|=3$. 
    If $|\mathcal{K}_4|=3$, then $|V(G)|=12$, contradicting the assumption that $|V(G)|\ge14$. 
    So $|\mathcal{K}_4|\le2$. 
    By \eqref{eq:3-con} and \eqref{eq:4}, $|E(G)|\le\frac{5|V(G)|}{2}-15+|\mathcal{K}_4|\le\frac{5|V(G)|}{2}-13$, with equality only if $|\mathcal{K}_4|=2$. 
    Therefore, the claim holds. 
    \end{proof}  
    
    Now we show the second part of the theorem. The sufficiency is implied by Lemma \ref{lem:mathcalG}. Next we only consider the necessity. Assume that $|E(G)|=(5|V(G)|-26)/2$. 
    By the Claim, $|B|=3$, $|\mathcal{K}_4|=2$ and
    $E_N(H(G,B))=\emptyset$.
    Thus, $G'_1,G'_2\in\mathcal{K}_4$.
    Since $G-B$ has exactly three components,
    $|V(Q_1)|=|V(Q_2)|=3$ and $|V(G)|\ge14$, we have $r=3$ and $|V(Q_3)|=|V(G)|-9\ge5$. 
    By Corollary \ref{cor:NR_TC}, $E(G)=\cup_{i=1}^{3}E_N(G'_i)$. 
    Since $E(G'_1)$, $E(G'_2)$ and $E(G'_3)$ form a partition of $E(G)$, we have $E(G'_i)=E_N(G'_i)$ for each $i\in\{1,2,3\}$, i.e., each $G'_i$ is minimal matching covered.
    In particular, $G'_1\cong G'_2\cong K_4$, $|V(G'_3)|=|V(Q_3)|+1\ge6$ and $|E(G'_3)|=|E(G)|-12=\frac{5|V(G'_3)|-10}{2}$.  
    By Theorem \ref{thm:main-thm}, $G'_3\cong {G}_k$ for some $k\ge2$ (see the definition of ${G}_k$ in Section 1). 
    If $\overline{q_3}\notin\{a,b\}$, then there exists some $i\in \{1,\ldots,k\}$,  $k\ge2$, such that $\overline{q_3}\in\{u_i,v_i\}$. So there exists $j\in\{1,\ldots,k\}\setminus\{i\}$ such that $u_jv_j$ forms a $K_2$-component of $G-\{a,b\}$, contradicting  that $G$ is 3-connected. 
   Hence  $\overline{q_3}\in\{a,b\}$, say  $\overline{q_3}=b$. 
    Thus, $Q_3={G}_k-b$.   
    Since $G$ is 3-connected, we have $N_G(V(Q_i))=B$ for each $i\in\{1,2,3\}$. 
    Moreover, since $G'_1\cong G'_2\cong K_4$ and $N_G(V(Q_i))=B$ for each $i\in\{1,2\}$, $Q_1$ and $Q_2$ are triangles, and $\partial_G(V(Q_i))$ is a matching for each $i\in\{1,2\}$.  

    Since $G'_3\cong {G}_k$ and $\overline{q_3}=b$, each vertex of $\{u_i,v_i:1\le i\le k\}$ has exactly one neighbor in $B$. 
    If $u_i$ and $v_i$ have the same neighbor $s\in B$ for some $i\in\{1,\ldots,k\}$, then $G[\{u_i,v_i\}]$ is a component of $G-\{a,s\}$, contradicting the assumption that $G$ is 3-connected. 
    Thus, $u_i$ and $v_i$ have distinct neighbors in $B$ for every $i\in\{1,\ldots,k\}$.
    Therefore, $G\in\mathcal{G}_k$.  
$\hfill\qedsymbol$


\begin{thebibliography}{99}

\bibitem{BM08} J. A. Bondy, U. S. R. Murty, Graph Theory, Springer-Verlag, Berlin, 2008.


\bibitem{CLM99} M. H. Carvalho, C. L. Lucchesi, U. S. R. Murty, Ear decompositions of matching covered graphs, Combinatorica  19 (2) (1999) 151--174.





 \bibitem{CLM15} M. H. Carvalho, C. L. Lucchesi, U. S. R. Murty, Thin edges in braces, Electron. J. Combin.  22 (4) (2015) \#P4.14.


\bibitem{ELP82} J. Edmonds, L. Lov\'asz, W. R. Pulleyblank, Brick decompositions and the matching rank of graphs, Combinatorica  2 (3) (1982) 247--274.

\bibitem{FKC2020} P. A. Fabres, N. Kothari,
M. H. Carvalho, Minimal braces, J. Graph Theory 96 (4) (2021) 490--509.

\bibitem{HLX2026} X. He, F. Lu, J. Xue, Wheel-like bricks and minimal matching covered graphs, J. Graph Theory 111 (2026) 5--16.

\bibitem{Hetyei1964} G. Hetyei, 2 $\times$ 1-es t\'eglalapokkal lefedhet\H{o} idomokr\'ol,  P\'ecsi Tan\'ark\'epz\H{o} F\H{o}isk. Tud. K\"ozl. (1964)  351--368 (in Hungarian).

\bibitem{LP77}
L. Lov\'{a}sz, M. D. Plummer, On minimal elementary bipartite graphs,
J. Combin. Theory, Ser. B  23 (1977) 127--138.


\bibitem{LP86} L. Lov\'{a}sz, M. D. Plummer, Matching Theory, Annals of Discrete Mathematics, vol. 29, Elsevier Science, 1986.

\bibitem{Lovasz87} L. Lov\'{a}sz, Matching structure and the matching lattice, J. Combin. Theory, Ser. B  43 (1987) 187--222.


\bibitem{Lucchesi2024}
C. L. Lucchesi, U. S. R. Murty, Perfect Matchings, Springer, 2024.






\bibitem{MDK2026} A. K. Mallik, A. A. Diwan, N. Kothari, Extremal minimal
bipartite matching covered graphs, J. Combin. 17 (4) (2026) 559--606.

\bibitem{NT2006}  S. Norine, R. Thomas, Minimal bricks, J. Comb. Theory Ser. B 96 (4) (2006) 505--513.



\bibitem{Tutte47} W. T. Tutte, The factorization of linear graphs, J. Lond. Math. Soc. 22 (1947) 107--111.

\bibitem{ZWY2022} Y. Zhang, X. Wang, J. Yuan, Bicritical graphs without removable edges, Discrete
Applied Mathematics 320 (2022) 1--10.

\end{thebibliography}
\end{document}